\documentclass[12pt]{amsart}
\usepackage{amscd,amsmath,amsthm,amssymb}
\usepackage[left]{lineno}
\usepackage{pstcol,pst-plot,pst-3d}
\usepackage{color}
\usepackage{pstricks}
\usepackage{stmaryrd}
\newpsstyle{fatline}{linewidth=1.5pt}
\newpsstyle{fyp}{fillstyle=solid,fillcolor=verylight}
\definecolor{verylight}{gray}{0.97}
\definecolor{light}{gray}{0.9}
\definecolor{medium}{gray}{0.85}
\definecolor{dark}{gray}{0.6}

 \def\NZQ{\mathbb}               
 \def\NN{{\NZQ N}}

 \def\FF{{\NZQ F}}

 \def\frk{\mathfrak}               

 \def\mm{{\frk m}}

 \def\G{{\mathcal G}}

 \def\bb{{\mathbf b}}
 \def\xb{{\mathbf x}}

 \def\cb{{\mathbf c}}

 \def\opn#1#2{\def#1{\operatorname{#2}}} 
 \opn\chara{char} \opn\length{\ell} \opn\pd{pd} \opn\rk{rk}
 \opn\projdim{proj\,dim} \opn\injdim{inj\,dim} \opn\rank{rank}
 \opn\depth{depth} \opn\grade{grade} \opn\height{height}
 \opn\embdim{emb\,dim} \opn\codim{codim}
 
 \opn\Tr{Tr} \opn\bigrank{big\,rank}
 \opn\superheight{superheight}\opn\lcm{lcm}
 \opn\trdeg{tr\,deg}
 \opn\reg{reg} \opn\lreg{lreg} \opn\ini{in} \opn\lpd{lpd}
 \opn\size{size} \opn\sdepth{sdepth}
 \opn\link{link}\opn\fdepth{fdepth}\opn\lex{lex}
 \opn\tr{tr}
 \opn\type{type}
 \opn\gap{gap}
 \opn\arithdeg{arith-deg}
 \opn\astab{astab}
  \opn\dstab{dstab}
  \opn\pol{pol}
 \opn\div{div} \opn\Div{Div} \opn\cl{cl} \opn\Cl{Cl}
 \opn\Spec{Spec} \opn\Supp{Supp} \opn\supp{supp} \opn\Sing{Sing}
 \opn\Ass{Ass} \opn\Min{Min}\opn\Mon{Mon}
 \opn\Ann{Ann} \opn\Rad{Rad} \opn\Soc{Soc}
 \opn\Im{Im} \opn\Ker{Ker} \opn\Coker{Coker} \opn\Am{Am}
 \opn\Hom{Hom} \opn\Tor{Tor} \opn\Ext{Ext} \opn\End{End}
 \opn\Aut{Aut} \opn\id{id}
 
 \opn\nat{nat}
 \opn\pff{pf}
 \opn\Pf{Pf} \opn\GL{GL} \opn\SL{SL} \opn\mod{mod} \opn\ord{ord}
 \opn\Gin{Gin} \opn\Hilb{Hilb}\opn\sort{sort}
 \opn\PF{PF}\opn\Ap{Ap}
 \opn\mult{mult}
 \opn\aff{aff}
 \opn\relint{relint} \opn\st{st}
 \opn\lk{lk} \opn\cn{cn} \opn\core{core} \opn\vol{vol}  \opn\inp{inp} \opn\nilpot{nilpot}
 \opn\link{link} \opn\star{star}\opn\lex{lex}\opn\set{set}
 \opn\width{wd}
 \opn\Fr{F}
 \opn\QF{QF}
 \opn\G{G}
 \opn\type{type}\opn\res{res}
 \opn\conv{conv}
 \opn\gr{gr}
 
 \def\pot#1#2{#1[\kern-0.28ex[#2]\kern-0.28ex]}

 \opn\dirlim{\underrightarrow{\lim}}
 \opn\inivlim{\underleftarrow{\lim}}
 \let\union=\cup

 \def\Implies{\ifmmode\Longrightarrow \else
         \unskip${}\Longrightarrow{}$\ignorespaces\fi}
 \def\implies{\ifmmode\Rightarrow \else
         \unskip${}\Rightarrow{}$\ignorespaces\fi}
 \def\iff{\ifmmode\Longleftrightarrow \else
         \unskip${}\Longleftrightarrow{}$\ignorespaces\fi}

 \let\:=\colon
 \newtheorem{Theorem}{Theorem}
 
 \newtheorem{Corollary}[Theorem]{Corollary}

 \let\epsilon\varepsilon
 \let\kappa=\varkappa
 \def\qed{\ifhmode\textqed\fi
       \ifmmode\ifinner\quad\qedsymbol\else\dispqed\fi\fi}
 \def\textqed{\unskip\nobreak\penalty50
        \hskip2em\hbox{}\nobreak\hfil\qedsymbol
        \parfillskip=0pt \finalhyphendemerits=0}
 \def\dispqed{\rlap{\qquad\qedsymbol}}
 
 \opn\dis{dis}
 \def\pnt{{\raise0.5mm\hbox{\large\bf.}}}
 
 \opn\Lex{Lex}

\begin{document}

\title {An upper bound for the regularity of powers of edge ideals}

\author {}

\address{J\"urgen Herzog, Fachbereich Mathematik, Universit\"at Duisburg-Essen, Campus Essen, 45117
Essen, Germany} \email{juergen.herzog@uni-essen.de}

\address{Takayuki Hibi, Department of Pure and Applied Mathematics, Graduate School of Information Science and Technology,
Osaka University, Toyonaka, Osaka 560-0043, Japan}
\email{hibi@math.sci.osaka-u.ac.jp}

\dedicatory{ }

\begin{abstract}
For a finite simple graph $G$ we give an upper  bound for the regularity of the powers of the edge ideal $I(G)$.
\end{abstract}

\thanks{}




\maketitle

\setcounter{tocdepth}{1}

In this note we  provide an upper  bound for the regularity of the powers of the edge ideal $I(G)$  of a finite simple graph $G$. A general lower bound is known by Beyarslan, H\`{a} and  Trung, see \cite[Theorem 4.5]{BHT}, while upper bounds are only known under additional assumptions, for example when $G$ is bipartite (\cite[Theorem 1.1]{JNS}. By Cutkosky, Herzog and Trung \cite{CHT} and Kodiyalam \cite{Ko} it is known that for  any graded ideal $I$ in $S=k[x_1,\ldots,x_n]$ there exist integers $a>0$ and $b\geq 0$ such that  $\reg I^s =as+b$ for all $s\gg 0$. In the case that $I$ is the edge ideal of a graph,  the constant $a$ is equal to $2$, so that $\reg I(G)^s=2s+b$ for $s\gg 0$. This result implies that there exists an integer  $c$ with  $\reg I(G)^s\leq 2s+c$ for all $s$.

In the following theorem we will see that one can choose  $c$  to be the dimension of the complex $\Delta(G)$ of stable sets of $G$. Recall that a subset $S$ of the vertex set $V(G)$ of $G$ is called a stable  (or independent) set, if no $2$-element subset of $G$ is an edge of $G$.

\begin{Theorem}
\label{main}
Let $G$ be a finite simple graph, and let $c$ be the dimension of the stable complex of $G$.  Then
\[
\reg I(G)^s\leq 2s+c \quad \text{for all } \quad s.
\]
\end{Theorem}

\begin{proof} The proof depends very much on a result  by Jayanthan and Selvaraja \cite{JS} for very well-covered graphs  which says that for any very well-covered graph  $G$ one has $\reg I(G)^s=s+\nu(G)-1$ for all $s\geq 1$, where $\nu(G)$ is the induced matching number of $G$. The same result was proved before by Norouzi, Seyed Fakhari, and Yassemi \cite{NFY} with an additional assumption.

Here we apply their theorem to the whisker graph $G^*$ of $G$. Recall that $G^*$ is obtained from $G$ by adding to each vertex of $G$ a leaf.  Thus if $V(G)=[n]$, then  $I(G^*)=(I(G), x_1y_1,\ldots,x_ny_n)\subset K[x_1,\ldots,x_n,y_1,\ldots,y_n]$. It is obvious that $G^*$ is a very well-covered graph. Indeed, since it is the polarization of the ideal $J=(I(G),x_1^2,\ldots,x_n^2)$,  it is a Cohen-Macaulay ideal, so that all maximal stable sets of $G^*$ have the same cardinality. On the other hand,  the vertices of $G$ form a maximal stable sets  of $G^*$.  This shows that all maximal stable sets  of $G$ have cardinality $n =|V(G^*)|/2$. Being very well-covered means exactly this. Thus by the above mentioned theorem Jayanthan et al we have $\reg I(G^*)^s=2s+\nu(G*)-1$.

Next we use the restriction lemma as given in \cite[Lemma 4.4]{HH}: let $I$ be a monomial ideal with multigraded (minimal) free resolution $\FF$,  and let $\cb\in \NN_\infty^n$, where $\NN_\infty=\NN\union\{\infty\}$. Then the restricted complex $\FF^{\leq \cb}$,  which is the subcomplex of $\FF$ for which $(\FF^{\leq\cb})_i$ is spanned by those basis elements of $\FF_i$ whose multidegree is componentwise less than or equal to $\cb$, is a (minimal) free resolution of the monomial ideal $I^{\leq\cb}$ which is generated by all monomials $\xb^\bb\in I$ with $\bb\leq \cb$, componentwise.

We choose $\cb=(\infty,\ldots,\infty,0,\ldots, 0)\in \NN_\infty^{2n}$ with $n$ components equal to $\infty$ and $n$ components equal to $0$. Then $(I(G^*)^s)^{\leq \cb}=I(G)^s$ for all $s$. Hence the restriction lemma implies that $\reg I(G)^s\leq \reg I(G^*)^s=2s+\nu(G^*)-1$ for all $s$. It remains to be shown that $\nu(G^*)=c+1$, where $c=\dim \Delta(G^*)$. In fact, since  $\ell =x_1-y_1,\ldots,x_n-y_n$ is a regular sequence and since $S/J$ is isomorphic to $K[x_1,\ldots,x_n,y_1,\ldots,y_n]/I(G^*)$ modulo $\ell$, it follows that $\reg J=\reg I(G^*)$.
Let  $\sigma$ be the maximal degree of a socle element of  $S/J$, then $\reg(J)=\deg \sigma +1$. The desired conclusion follows since $S/J$ has a $K$-basis  consisting elements $u+J$, where  $u=\prod_{i\in F}x_i$  and $F\in \Delta(G)$.
\end{proof}

\begin{Corollary}
\label{maybe}
Let $G$ be a finite simple  graph with $n$ vertices and $e$ edges. Then
\[
\reg I(G)^s\leq 2s+ \lfloor 1/2+\sqrt{1/4+n^2-n-2e}\rfloor-1   \quad \text{for all } \quad s.
\]
\end{Corollary}

\begin{proof}
For the proof we use Theorem~\ref{main} and a famous formula of Hansen \cite{H} who showed that the  size  of a maximal stable set is bounded by $\lfloor 1/2+\sqrt{1/4+n^2-n-2e}\rfloor$.
\end{proof}

There are many other upper bounds for the size  of a maximal stable set of a graph. Well known is the bound given by Kwok  which is  given as an exercise in \cite{W}.  Kwok's upper bound is $n-e/\Delta$,  where $\Delta$ is is the maximal degree of a vertex of $G$. A survey on the known upper bounds can be found in the thesis of Willis \cite{WW}.

\medskip
Even though  $f(s)=\reg I^s$  is linear function of $s$ for $s\gg 0$ when $I$ is a graded ideal of the polynomial ring, the initial behaviour of $f(s)$ is not so well understood. In \cite{C}, Conca gives some examples for the unexpected behaviour of the function $f(s)$. On the positive side, Eisenbud and Harris \cite[Proposition 1.1]{EH}  showed that for a graded ideal $I\subset S =K[x_1,\ldots,x_n]$ with  $\height I=n$ which is generated in a single degree, say $d$, one has $f(s)=ds+b_i$ with $b_1\geq b_2\geq \cdots$. We will use this result in the proof of the next theorem.

For a monomial ideal $L\subset S$ we denote by $L^{\pol}$ the polarization of $L$, and by $S^{\pol}$ the polynomial ring in the variables which are needed to define $L^{\pol}$.

\begin{Theorem}
\label{equal}
Let $G$ be a finite simple graph with $n$ vertices, and let $$J=(I(G),x_1^2,\ldots,x_n^2).$$ Then $\reg I(G)^s\leq\reg J^s=\reg (J^{\pol})^s$ for all $s$.
\end{Theorem}

\begin{proof}
The inequality $\reg I(G)^s\leq\reg J^s$ follows from the equality $\reg J^s=\reg (J^{\pol})^s$ and the proof of Theorem~\ref{main}. Thus it remains to prove these equalities. For $s=1$, the equality holds, since polarization of an ideal does not change its graded Betti numbers. Now since $J^{\pol}=I(G^*)$,  \cite{JS}  implies that $\reg (J^{\pol})^s-2s$ is a constant function on $s$. Thus the desired result follows once we have shown that $\reg J^s-2s$ is also a constant function on $s$. Indeed we will show that $\reg J^{s} -\reg J\geq 2(s-1)$ for all $s\geq 1$. Then,  together with the result of Eisenbud and Harris,   the desired conclusion follows.

In order to prove  $\reg J^{s}- \reg J\geq  2(s-1)$ for all $s\geq 1$, we show  the following: let $F\in \Delta(G)$ a facet with $|F|=c+1$, and set $u=\prod_{i\in F}x_i$. We may assume that $x_1$ divides $u$, and consider $w=ux_1^{2(s-1)}$. Let $\mm=(x_1,\ldots,x_n)$. Since $\mm u\in J$ it follows that $\mm w\in J^s$. It remains to be shown that $w\not\in J^s$. Indeed, suppose that $w=w_1w_2\cdots w_s$ with $w_i\in J$ for $i=1,\ldots, s$. Since $x_1^{2s}$ does not divide $w$, one of the factors, say  $w_1$,  must be squarefree. Since $w_1$ divides $w$, it then follows that $w_1\not\in J$, a contradiction.
\end{proof}
It should be noted that the equalities  $\reg J^s=\reg (J^{\pol})^s$ are quite surprising because in general polarization and taking powers are not very well compatible operations.

\end{document}